\documentclass[reqno]{amsart}
\usepackage{amssymb,latexsym,amsmath,amsthm,enumerate}
\usepackage[mathscr]{eucal}
\usepackage{framed,color,graphicx}
\usepackage{mathrsfs}
\usepackage[all]{xy}
\usepackage{tikz}
\usepackage{cite}

\usetikzlibrary{positioning}
\tikzset{%
element/.style={draw, shape=circle, fill=white, inner sep=1.4pt}
}

\DeclareSymbolFont{bbold}{U}{bbold}{m}{n}
\DeclareSymbolFontAlphabet{\mathbbold}{bbold}

\theoremstyle{plain}
\newtheorem{thm}{Theorem}[section]
\newtheorem{lem}[thm]{Lemma}
\newtheorem{cor}[thm]{Corollary}
\newtheorem{pro}[thm]{Proposition}

\newtheorem{problem}[thm]{Problem}

\newtheorem{claim}{Claim}[section]

\theoremstyle{definition}

\newcommand{\up}[1]{\textup{#1}}

\newcommand{\bp}{\mathbf{p}}
\newcommand{\bq}{\mathbf{q}}
\newcommand{\br}{\mathbf{r}}
\newcommand{\bs}{\mathbf{s}}
\newcommand{\bt}{\mathbf{t}}
\newcommand{\bu}{\mathbf{u}}
\newcommand{\bv}{\mathbf{v}}
\newcommand{\bw}{\mathbf{w}}

\begin{document}

\title[The additively idempotent semirings]
{Two nonfinitely based additively idempotent semirings of order four}

\author{Mengya Yue}
\address{School of Mathematics, Northwest University, Xi'an, 710127, Shaanxi, P.R. China}
\email{myayue@yeah.net}

\author{Miaomiao Ren}
\address{School of Mathematics, Northwest University, Xi'an, 710127, Shaanxi, P.R. China}
\email{miaomiaoren@yeah.net}

\author{Zidong Gao}
\address{School of Mathematics, Northwest University, Xi'an, 710127, Shaanxi, P.R. China}
\email{zidonggao@yeah.net}

\thanks{Miaomiao Ren, corresponding author, is supported by National Natural Science Foundation of China (12371024, 12571020).
}

\begin{abstract}
We establish two sufficient conditions for an additively idempotent semiring to be nonfinitely based.
As applications,
we prove that two specific $4$-element additively idempotent semirings,
$S_{(4,545)}$ and $S_{(4,634)}$, whose additive reducts are chains,
have no finite basis for their identities.
Furthermore, we show that the interval $[\mathsf{V}(S_{(4,545)}),\mathsf{V}(S_{(4,634)})]$
in the lattice of semiring varieties contains \(2^{\aleph_0}\) distinct varieties.
Consequently, the join of two finitely based additively idempotent semiring varieties is not necessarily finitely based.
Moreover, we obtain the smallest example of a finitely based additively idempotent semiring $S$
whose extension $S^0$ (obtained by adjoining a new element) is nonfinitely based.
\end{abstract}

\maketitle

\section{Introduction}
An \emph{additively idempotent semiring} (or \emph{ai-semiring} for short)
is an algebra $(S, +, \cdot)$ equipped with two binary operations $+$ and $\cdot$ satisfying the following axioms:
\begin{itemize}
\item The additive reduct $(S, +)$ forms a commutative idempotent semigroup;

\item The multiplicative reduct $(S, \cdot)$ forms a semigroup;

\item The distributive laws hold:
\[
x(y+z)\approx xy+xz,~\text{and}~(x+y)z\approx xz+yz.
\]
\end{itemize}
An ai-semiring whose multiplicative reduct is commutative is called a \emph{commutative ai-semiring}.

The class of ai-semirings includes
well-known examples such as the Kleene semiring of regular languages~\cite{con},
the max-plus algebra~\cite{aei}, the power semiring of a semigroup~\cite{dgv24},
the endomorphism semiring of a semilattice~\cite{dgv25},
the semiring of all binary relations on a set \cite{dolinka}, and distributive lattices~\cite{bs}.
These and other similar algebras have significant applications in various fields,
including algebraic geometry~\cite{cc}, tropical algebraic geometry~\cite{ms},
theoretical computer science~\cite{go}, information science~\cite{gl}, and the theory of weighted automata~\cite{dkv}.

Given an ai-semiring $S$, the binary relation $\leq_S$ defined by
\[
a \leq_S b \Leftrightarrow a+b=b,
\]
is a partial order on $S$.
One readily checks that $\leq_S$ is compatible with both addition and multiplication.
For this reason an ai-semiring is often called a \emph{semilattice-ordered semigroup}.
Whenever an order on an ai-semiring is mentioned, it always refer to the order defined above.

A class of ai-semirings is a \emph{variety} if it is closed under taking subalgebras,
homomorphic images, and arbitrary direct products.
By Birkhoff's theorem, a class of ai-semirings is variety if and only if
it is an equational class, that is, the class of all ai-semirings satisfying a certain set of identities.
Let $\mathcal{V}$ be an ai-semiring variety.
If $\Sigma$ is a set of identities that defines $\mathcal{V}$,
then $\Sigma$ is an \emph{equational basis} of $\mathcal{V}$.
The variety $\mathcal{V}$ is \emph{finitely based} if it admits a finite equational basis;
otherwise, it is \emph{nonfinitely based}.

For an ai-semiring $S$, let $\mathsf{V}(S)$ denote the variety generated by $S$,
that is, the smallest variety containing $S$.
Then $S$ and $\mathsf{V}(S)$ satisfy exactly the same identities.
If $\Sigma$ is an equational basis of $\mathsf{V}(S)$, we also say that $\Sigma$ is is an equational basis of $S$.
The ai-semiring $S$ is called finitely based or nonfinitely based according to whether $\mathsf{V}(S)$ is finitely based or not.

The \emph{finite basis problem} for a class of ai-semirings,
one of the central questions in the theory of varieties,
concerns the classification of its members with respect to the finite basis property.
While we focus here on ai-semirings, the finite basis problem for arbitrary algebras
has deep and often surprising connections with
formal languages \cite{Almeida}, the complexity of cellular automata \cite{Gulak} and
classical number-theoretic conjectures \cite{Perkins}.

Over the past two decades, the finite basis problem for ai-semirings has been intensively studied and well developed,
for example, see~\cite{do, dolinka, dgv24, dgv25, gmrz, gpz, gv2501, jrz, pas05, rjzl, rlzc, rlyc, rz16, rzs20, rzw, sr, shap23, Volkov, wrz, yrzs, zrc, zw}.
In particular, the problem for ai-semirings of small order has received special attention in the literature,
as their classification often serves as a testing ground for new methods and supplies crucial examples for the general theory.

Dolinka~\cite{do} found the first example of a nonfinitely based finite ai-semiring, which has $7$ elements.
Dolinka~\cite{dolinka} later presented a $16$-element nonfinitely based ai-semiring.
Volkov~\cite{Volkov} proved that the ai-semiring $B_2^1$ whose multiplicative reduct is a
$6$-element Brandt monoid has no finite equational basis.
Gusev and Volkov~\cite{gv2501} provided another $6$-element nonfinitely based ai-semiring.

Shao and Ren~\cite{sr} established that every ai-semiring in the
variety generated by all ai-semirings of order two is finitely based.
There are, up to isomorphism, $61$ ai-semirings of order three, which are denoted by $S_i$, $1\leq i \leq 61$.
Zhao et al.~\cite{zrc} showed that all ai-semirings of order three are finitely based,
with the sole possible exception of the semiring $S_7$ (its Cayley tables are given in Table~\ref{tb24111401}).
Jackson et al.~\cite{jrz} later provided general sufficient conditions
for a finite ai-semiring to be nonfinitely based,
and applied one such condition to confirm that $S_7$ itself is nonfinitely based.
(We note that $S_7$ can be embedded into $B_2^1$.)
This completes the classification of ai-semirings of order at most three
with respect to the finite basis property.

Moreover, Jackson et al.~\cite{jrz} and Gao et al.~\cite{gmrz, gjrz2}
demonstrated that the nonfinite basis property of $S_7$ transfers to
many finite ai-semirings whose varieties contain it.
This observation led Gao et al.~\cite{gjrz2} to propose the following general problem:
\begin{problem}\label{prob0127}
Is every finite ai-semiring whose variety contains $S_7$ nonfinitely based$?$
\end{problem}

\begin{table}[ht]
\caption{The Cayley tables of $S_7$} \label{tb24111401}
\begin{tabular}{c|ccc}
$+$      &$\infty$&$a$&$1$\\
\hline
$\infty$ &$\infty$&$\infty$&$\infty$\\
$a$      &$\infty$&$a$&$\infty$\\
$1$      &$\infty$&$\infty$&$1$\\
\end{tabular}\qquad
\begin{tabular}{c|ccc}
$\cdot$  &$\infty$&$a$&$1$\\
\hline
$\infty$ &$\infty$&$\infty$&$\infty$\\
$a$      &$\infty$&$\infty$&$a$\\
$1$      &$\infty$&$a$&$1$\\
\end{tabular}
\end{table}

\begin{table}[ht]
\caption{The Cayley tables of $S_{53}$} \label{tb53}
\begin{tabular}{c|ccc}
$+$      &$\infty$&$a$&$1$\\
\hline
$\infty$ &$\infty$&$\infty$&$\infty$\\
$a$      &$\infty$&$a$&$a$\\
$1$      &$\infty$&$a$&$1$\\
\end{tabular}\qquad
\begin{tabular}{c|ccc}
$\cdot$  &$\infty$&$a$&$1$\\
\hline
$\infty$ &$\infty$&$\infty$&$\infty$\\
$a$      &$\infty$&$\infty$&$a$\\
$1$      &$\infty$&$a$&$1$\\
\end{tabular}
\end{table}

Recall that $S_{53}$ denotes another $3$-element ai-semiring whose Cayley tables given in Table~\ref{tb53}.
Although $S_{53}$ and $S_7$ share the same multiplicative reduct, their additive reducts are different.
This structural similarity, however, does not imply any inclusion between the varieties they generate.
In fact, the varieties $\mathsf{V}(S_{53})$ and $\mathsf{V}(S_{7})$ are incomparable (neither is contained in the other),
and they are known to have different finite basis properties.

Up to isomorphism, there are exactly $866$ ai-semirings of order four,
which are denoted by $S_{(4, i)}$, $1\leq i\leq 866$.
These algebras are categorized into five distinct types based on their additive orders,
as illustrated in \cite[Figure 1]{rlzc}. The finite basis problem for ai-semirings in the first three types has
been resolved by Gao et al.~\cite{gmrz}, Ren et al.~\cite{rjzl, rlyc, rlzc}, Shaprynski\v{\i}~\cite{shap23},
Wu et al.~\cite{wrz}, and Yue et al.~\cite{yrzs}.
Recently, members of our group have begun to investigate the problem for the largest of the remaining types,
namely the class $\{S_{(4,i)} \mid 481\leq i\leq 866\}$,
which consists of $386$ ai-semirings of order four whose additive reducts are chains.

\begin{table}[ht]
\caption{The Cayley tables of $S_{(4, 545)}$} \label{tb545}
\begin{tabular}{c|cccc}
$+$      &$\infty$&$a$&$1$&$0$\\
\hline
$\infty$ &$\infty$&$\infty$&$\infty$&$\infty$\\
$a$      &$\infty$&$a$&$a$&$a$\\
$1$      &$\infty$&$a$&$1$&$1$\\
$0$      &$\infty$&$a$&$1$&$0$\\
\end{tabular}\qquad
\begin{tabular}{c|cccc}
$\cdot$  &$\infty$&$a$&$1$&$0$\\
\hline
$\infty$ &$\infty$&$\infty$&$\infty$&$\infty$\\
$a$      &$\infty$&$\infty$&$a$&$a$\\
$1$      &$\infty$&$a$&$1$&$0$\\
$0$       &$\infty$&$a$&$0$&$0$\\
\end{tabular}
\end{table}

\begin{table}[ht]
\caption{The Cayley tables of $S_{(4, 634)}$} \label{tb634}
\begin{tabular}{c|cccc}
$+$      &$\infty$&$a$&$1$&$0$\\
\hline
$\infty$ &$\infty$&$\infty$&$\infty$&$\infty$\\
$a$      &$\infty$&$a$&$a$&$a$\\
$1$      &$\infty$&$a$&$1$&$1$\\
$0$      &$\infty$&$a$&$1$&$0$\\
\end{tabular}\qquad
\begin{tabular}{c|cccc}
$\cdot$  &$\infty$&$a$&$1$&$0$\\
\hline
$\infty$ &$\infty$&$\infty$&$\infty$&$0$\\
$a$      &$\infty$&$\infty$&$a$&$0$\\
$1$      &$\infty$&$a$&$1$&$0$\\
$0$        &$0$&$0$&$0$&$0$\\
\end{tabular}
\end{table}

The present paper concentrates on two algebras from this type: $S_{(4, 545)}$ and $S_{(4, 634)}$,
whose Cayley tables are presented in Tables \ref{tb545} and \ref{tb634}.
While both algebras have the same additive reduct (a $4$-element chain), their multiplicative reducts are distinct.
We shall prove that $\mathsf{V}(S_{(4, 545)})$ is a subvariety of $\mathsf{V}(S_{(4, 634)})$.
Moreover, we demonstrate that $\mathsf{V}(S_{(4, 545)})$ contains both $S_7$ and $S_{53}$.

Even though $S_{(4,545)}$ and $S_{(4,634)}$ are merely $4$-element ai-semirings with commutative multiplication,
which might suggest they are amenable to existing methods,
all known sufficient conditions in the literature for deciding the finite basis property of a finite ai-semiring fail for these two algebras.

We first develop two sufficient conditions for an ai-semiring to be nonfinitely based,
and then apply them to solve the finite basis problem for $S_{(4,545)}$ and $S_{(4,634)}$.
In doing so we provide an explicit infinite equational basis for each of them,
employing syntactic approach together with the compactness theorem of equational logic (see Lemma~\ref{el26012701}).
The main result of this paper is both $S_{(4, 545)}$ and $S_{(4, 634)}$ are nonfinitely based,
which constitutes a contribution to Problem~\ref{prob0127}.

Although substantial progress has been made on the finite basis question for ai-semirings,
the following fundamental problem remains open.
\begin{problem}\label{problem1}
Is the join of two finitely based ai-semiring varieties again finitely based?
\end{problem}

We settle Problem~\ref{problem1} negatively by showing that $\mathsf{V}(S_{(4, 545)})$
is a nonfinitely based variety that can be expressed as
the join of two finitely based varieties.

Let $S$ be an ai-semiring.
The algebra $S^0 = S \cup \{0\}$ obtained from $S$ by adjoining a new element $0$ is again an ai-semiring under the operations
\[
(\forall a\in S\cup \{0\}) \quad a+0=0+a=a,\quad a0=0a=0,
\]
with the original operations of $S$ unchanged.
Thus $S$ is a subsemiring of $S^0$, and the element
$0$ is simultaneously the additive least element and the multiplicative zero of $S^0$.
One can easily check that $S_{(4, 634)}$ is isomorphic to $S^0_{53}$,
and so
\[
\mathsf{V}(S_{(4, 634)})=\mathsf{V}(S^0_{53}).
\]

When $S$ is the trivial ai-semiring, $S^0$ is the $2$-element distributive lattice $D_2$.
By \cite[Proposition 1.4]{wrz}, $D_2$ is contained in every ai-semiring variety of the form $\mathsf{V}(S^0)$.
We shall show that the variety $\mathsf{V}(S_{(4, 545)})$ is exactly the join of $\mathsf{V}(S_{53})$ and $\mathsf{V}(D_{2})$.
Hence $\mathsf{V}(S_{(4, 545)})$ is a subvariety of $\mathsf{V}(S_{(4, 634)})$.

The construction $S^0$ has played a central role in earlier studies of the finite basis problem for ai-semirings,
see, for example,~\cite{gpz, pas05, rz16, rzs20, rzw}.
For an ai-semiring $S$, $[\mathsf{V}(S), \mathsf{V}(S^0)]$ denotes the
the class of all varieties that are contained in $\mathsf{V}(S^0)$ and contains $\mathsf{V}(S)$.
Jackson et al.~\cite[Problem 7.4]{jrz} proposed the following problem:
\begin{problem}\label{prob123}
\hspace*{\fill}
\begin{enumerate}[$(1)$]
\item Is $S_7^0$ finitely based or not finitely based$?$
\item What is the cardinality of the interval $[\mathsf{V}(S_7),\mathsf{V}(S_7^0)]$ in the lattice of semiring varieties$?$
\item Under what conditions is the $($non$)$finite basis property preserved when passing from an ai-semiring $S$ to $S^0$$?$
\end{enumerate}
\end{problem}

Wu et al.~\cite{wrz} proved that $S_7^0$ is nonfinitely based.
Zhao and Wu~\cite{zw} exhibited a $4$-element finitely based ai-semiring $S$ for which the extension $S^0$ is nonfinitely based.
Recently, Gao et al.~\cite{gjrz2} showed that the interval $[\mathsf{V}(S_7),\mathsf{V}(S_7^0)]$ contains $2^{\aleph_0}$ distinct varieties.
The present paper contributes further to Problem~\ref{prob123} by
showing that $S_{53}^0$ is nonfinitely based.
Since $S_{53}$ is a finitely based ai-semiring of order three,
one can deduce that $S_{53}$
is the smallest ai-semiring $S$ (up to order) such that $S$ is finitely based while its extension $S^0$ is not.

Moreover,
Gao et al.~\cite{gjrz2} proposed the following general problem:
\begin{problem}\label{prob1123}
Let $S$ be an ai-semiring. What is the cardinality of the interval $[\mathsf{V}(S),\mathsf{V}(S^0)]$ in the lattice of semiring varieties$?$
\end{problem}

They proved that if $S$ is an ai-semiring such that
$\mathsf{V}(S)$ is contained in $\mathsf{V}(B_2^1)$ and contains $\mathsf{V}(S_7)$,
then the interval $[\mathsf{V}(S), \mathsf{V}(S^0)]$ has cardinality $2^{\aleph_0}$.
Here we show that the interval $[\mathsf{V}(S_{53}), \mathsf{V}(S_{53}^0)]$ likewise contains $2^{\aleph_0}$ distinct varieties.
Note that in the earlier examples, both ends of the interval are nonfinitely based;
in contrast,
our example satisfies that the left end $\mathsf{V}(S_{53})$
is finitely based while the right end $\mathsf{V}(S_{53}^0)$ is nonfinitely based.
This provides a further contribution to Problem~\ref{prob1123}.

\section{Preliminaries}
In this section, we collect some basic notation, definitions, and tools that will be used throughout the paper.
Although every ai-semiring in the subsequent sections is commutative,
the results quoted below are stated for general ai-semirings (whose multiplication need not be commutative).
This formulation keeps them applicable in future work that may not impose commutativity.

Let $X$ be a countably infinite set of variables,
let $X^+$ denote the free semigroup over $X$,
$X^*$ the free monoid over $X$ (with $\varepsilon$ denoting the empty word),
and $X_c^+$ the free commutative semigroup over $X$.

An \emph{ai-semiring term} (or simply a \emph{term}) over $X$ is a finite nonempty set of words in $X^+$.
(In what follows, terms are denoted by bold lowercase letters \(\mathbf{u}, \mathbf{v}, \mathbf{w}, \dots\),
while ordinary lowercase letters \(x, y, z, \dots\) stand for variables.)
We write a term as a formal sum of its elements; that is,
$\bw=\bu_1+\bu_2+\cdots+\bu_n$ means $\bw=\{\bu_1, \bu_2, \ldots, \bu_n\}$.
The order of the summands in the formal sum is irrelevant,
and repeated occurrences of the same word are identified with a single occurrence.
Two terms are equal if and only if their underlying sets coincide.

The collection of all terms over $X$, denoted by $P_f(X^+)$,
forms an ai-semiring under the usual term addition and multiplication.
By \cite[Theorem 2.5]{kp}, $P_f(X^+)$ is free in the variety of all ai-semirings over $X$.
The multiplication of $P_f(X^+)$ is cancellative: for any terms $\bp, \bq$ and $\br$, if $\bp\bq=\bp\br$, then $\bq=\br$;
if $\bq\bp=\br\bp$, then $\bq=\br$.
An \emph{ai-semiring substitution} (or simply a \emph{substitution}) is a semiring homomorphism from $P_f(X^+)$ to itself.

Let $\bu$ and $\bv$ be terms.
We say that \emph{$\bu$ is a subterm of $\bv$}, and write $\bu\leq \bv$,
if there exist terms $\bp_1, \bp_2, \bp_3$ such that
\[
\bv=\bp_1\bu\bp_2+\bp_3.
\]
Here $\bp_1$ or $\bp_2$ may be the empty word (acting as the multiplicative identity),
and $\bp_3$ may be the empty set (acting as the additive identity).
In particular, if $\bp_1$ and $\bp_2$ are both empty, then $\bu$ is called an \emph{additive subterm} of $\bv$;
if $\bp_3$ is empty, then $\bu$ is called a \emph{multiplicative subterm} of $\bv$.
If $\bv$ is a word and $\bu\leq \bv$,
then $\bu$ must be a multiplicative subterm of $\bv$ and is itself a word;
in this case we say that $\bu$ is \emph{subword} of $\bv$.
Moreover, the relation $\leq$ is a partial order on $P_f(X^+)$
and is compatible with multiplication.

The term $\bv$ is called \emph{$\bu$-free} if for every substitution $\varphi$,
the term $\varphi(\bu)$ is not a subterm of $\bv$.
The following statement follows directly from the definitions of subterm and freeness.

\begin{lem}\label{lem26012701}
Let $\bu$, $\bv$ and $\bw$ be terms.
If $\bv$ is $\bu$-free and $\bu$ is a subterm of $\bw$, then $\bv$ is also $\bw$-free.
\end{lem}

An \emph{ai-semiring identity} (or simply an \emph{identity}) is a formal expression of the form
$\bu\approx \bv$, where $\bu$ and $\bv$ are terms.
Let $S$ be an ai-semiring and $\bu\approx \bv$ an identity.
We say that $S$ \emph{satisfies} $\bu\approx \bv$, or that $\bu\approx \bv$ \emph{holds} in $S$,
if $\varphi(\bu)=\varphi(\bv)$ for every semiring homomorphism $\varphi\colon P_f(X^+) \rightarrow S$.

The following result concerns the equational logic of ai-semirings and is from \cite[Lemma 3.1]{dolinka}.
\begin{lem}\label{lem02}
Let $\Sigma$ be a set of identities and let $\bu\approx \bv$ be a nontrivial identity.
Then $\bu\approx \bv$ is derivable from $\Sigma$ if and only if there
exist terms $\bt_1, \bt_2, \dots, \bt_n\in P_f(X^+)$ such that $\bu=\bt_1$, $\bv=\bt_n$ and,
for each $1\leq i<n$, there are terms $\bp_i, \bq_i, \br_i, \bs_i, \bs'_{i} \in P_f(X^+)$ and a substitution
$\varphi_i\colon P_f(X^+) \to P_f(X^+)$ such that
\[
\bt_i = \bp_i\varphi_i(\bs_i)\bq_i+\br_i, \quad \bt_{i+1}=\bp_i\varphi_i(\bs'_{i})\bq_i+\br_i,
\]
where $\bs_{i} \approx \bs'_{i}\in\Sigma$ or $\bs'_{i} \approx \bs_{i} \in \Sigma$,
$\bp_i$ or $\bq_i$ may be the empty word \up(acting as the multiplicative identity\up),
and $\br_i$ may  be the empty set \up(acting as the additive identity\up).
\end{lem}

The following lemma, which is the compactness theorem of equational logic (see \cite[Exercise II.14.10]{bs}),
will be our main tool in Sections \ref{sec4} and \ref{sec5} for proving that both $S_{(4, 545)}$ and $S_{(4, 634)}$ have no finite equational basis.
\begin{lem}\label{el26012701}
Let $\mathcal{V}$ be an ai-semiring variety.
If $\mathcal{V}$ is finitely based,
then every equational basis of $\mathcal{V}$ contains a finite subset that also defines $\mathcal{V}$.
Equivalently, if $\mathcal{V}$ has an infinite equational basis none of whose finite subsets defines $\mathcal{V}$,
then $\mathcal{V}$ is nonfinitely based.
\end{lem}

Following Volkov~\cite{volkov2024},
we write $\bu \preceq \bv$ (or equivalently $\bv \succeq \bu$) to denote the identity $\bu+\bv \approx \bv$
and call it an \emph{ai-semiring inequality} (or simply an \emph{inequality}).
The term $\mathbf{u}$ is called the \emph{lower side} and $\mathbf{v}$ the \emph{upper side} of the inequality.
It is straightforward to verify that an ai-semiring $S$ satisfies an inequality $\bu\preceq \bv$
if and only if $\varphi(\bu)\leq\varphi(\bv)$ for every semiring homomorphism $\varphi\colon P_f(X^+) \rightarrow S$.
Consequently, $S$ satisfies an identity $\bu\approx \bv$
precisely when it satisfies both inequalities $\bu\preceq\bv$ and $\bv\preceq\bu$.
For this reason, every set of identities can be equivalently expressed as a set of inequalities.

Suppose that $\Sigma$ is a set of identities.
Let $\mathbf{u}\approx \mathbf{v}$ be an identity such that
\[
\mathbf{u}=\mathbf{u}_1+\cdots+\mathbf{u}_k,\quad \mathbf{v}=\mathbf{v}_1+\cdots+\mathbf{v}_\ell,
\]
where $\mathbf{u}_i,\mathbf{v}_j\in X^+$ for $1\leq i\leq k$ and $1\leq j\leq \ell$.
One can readily check that the ai-semiring variety defined by $\mathbf{u}\approx \mathbf{v}$
coincides with the ai-semiring variety defined by the inequalities
\[
\mathbf{u}_i\preceq \mathbf{v}, \quad \mathbf{v}_j\preceq \mathbf{u}\quad(1\leq i\leq k, ~1\leq j\leq \ell).
\]
Consequently, to prove that $\mathbf{u}\approx \mathbf{v}$ is derivable from $\Sigma$,
it suffices to show that for every $i$ and $j$,
the inequalities
$\mathbf{u}_i\preceq \mathbf{v}$, $\mathbf{v}_j\preceq \mathbf{u}$
can be derived from $\Sigma$.
For this reason, throughout this paper we restrict our attention to the inequalities of the form $\bq\preceq\bu$,
where $\bq$ is a word and $\bu$ is a term.
This technique will be used repeatedly in the sequel.

Next, we introduce some notation.
Let $\bw$ be a nonempty word, and let $x$ be a variable. Then
\begin{itemize}
\item $c(\bw)$ denotes the \emph{content} of $\bw$, that is, the set of all variables that occur in $\bw$;

\item $\ell(\bw)$ denotes the \emph{length} of $\bw$, that is, the number of variables occurring in $\bw$ counting multiplicities;

\item $S_2(\bw)$ denotes the set of all subwords of length $2$ of $\bw$;

\item $occ(x, \bw)$ denotes the number of occurrences of $x$ in $\bw$.
\end{itemize}

Now let $\bu$ be a term such that $\bu=\bu_1+\bu_2+\cdots+\bu_n$,
where $\bu_i \in X^+$, $1 \leq i \leq n$.
Let $\bq$ be a nonempty word, and let $k$ be a positive integer. Then
\begin{itemize}
\item $c(\bu)$ denotes the \emph{content} of $\bu$, that is,
\[
c(\bu)=\bigcup_{i=1}^n c(\bu_i);
\]

\item $L_{\geq k}(\bu)$ denotes the set $\{\bu_i \in \bu \mid \ell(\bu_i)\geq k\}$;

\item $D_{\bq}(\bu)$ denotes the set $\{\bu_i \in \bu \mid c(\bu_i)\subseteq c(\bq)\}$;

\item $S_2(\bu)$ denotes the set $\bigcup_{1\leq i \leq n}S_2(\bu_i)$;

\item $\delta(\bu)$ denotes the set of nonempty subsets $Z$ of $c(\bu)$ such that for every
 $\bu_i\in\bu$, $Z\cap c(\bu_i)$ is a singleton and $occ(x,\bu_i)=1$ if $\{x\}=Z\cap c(\bu_i)$.
\end{itemize}

Finally, we state the characterization of identities of $D_2$, which follows directly from Shao and Ren~\cite[Lemma 1.1]{sr}.
\begin{lem}\label{lem01}
Let $\bq\preceq\bu$ be a nontrivial inequality such that
$\bu=\bu_1+\bu_2+\cdots+\bu_n$ and $\bu_i, \bq \in X^+$ for all $1\leq i \leq n$. Then
$\bq\preceq\bu$ holds in $D_2$ if and only if $c(\bu_i) \subseteq c(\bq)$ for some $\bu_i \in \bu$.
\end{lem}

\section{Equational basis for $S_{(4,545)}$}\label{sec4}
In this section, we first establish a sufficient condition for an ai-semiring to be nonfinitely based.
As an application, we prove that $S_{(4,545)}$ has no finite equational basis.
Throughout the remainder of the paper
every term is built from words in the free commutative semigroup $X^+_c$;
that is, words are considered up to permutation of variables.

Let $n\geq 1$ be an integer, and let $0\leq k\leq n$. Define
\[
\bu_{n, k}=\prod_{i=1}^n x_i+\left(\sum_{i=1}^kx_ix_{n+1}\right)+\left(\sum_{i=k+1}^nx_ix_{n+1}y_i\right).
\]
In particular, if $k=0$, then the middle sum is absent, and so
\[
\bu_{n, 0}=\prod_{i=1}^n x_i+\left(\sum_{i=1}^nx_ix_{n+1}y_i\right);
\]
if $k=n$, then the last sum disappears, and so
\[
\bu_{n, n}=\prod_{i=1}^n x_i+\left(\sum_{i=1}^nx_ix_{n+1}\right).
\]
For convenience, we sometimes factor $x_{n+1}$ out of the two sums and write
\[
\bu_{n, k}=\prod_{i=1}^n x_i+\left(\sum_{i=1}^kx_i+\sum_{i=k+1}^nx_iy_i\right)x_{n+1}.
\]

The following lemma collects several basic properties of the terms $\mathbf{u}_{n,k}$,
which will be used repeatedly in the subsequent proofs.
\begin{lem}\label{lem26011801}
Let $n\geq 1$ be an integer, and let $0\leq k\leq n$.
\begin{itemize}
\item[(a)] Every word in $\bu_{n,k}$ is linear.

\item[(b)] The content intersection of any two distinct words in $\bu_{n,k}$ is a singleton.

\item[(c)] No distinct words in $\bu_{n,k}$ share a common subword of length $2$.

\item[(d)] $\bu_{n,k}$ has no subterm that is the square of a term.

\item[(e)] If $n\geq 2$ and $\bp$ and $\bq $ are words in $\bu_{n,k}$ with $\bp \leq\bq$, then $\bp=\bq$.

\item[(f)] $x_1x_2\cdots x_n$ is the unique additive subterm of $\bu_{n,k}$
whose content is contained in $\{x_1, x_2, \ldots, x_n\}$.

\item[(g)] If $n\geq 4$, then $x_1x_2\cdots x_n$ is the unique word in $\bu_{n, k}$ of length at least $4$.

\item[(h)]
If $n\geq 3$, then
the term $\bu_{n, n}$ contains exactly one word, $x_1x_2\cdots x_n$, whose length is at least $3$.

\item[(i)]
If $n\geq 3$, then
$\bu_{n, n}$ does not have a subterm of the form $\bt_1\bt_2+\bt_2\bt_3+\bt_3\bt_1$.

\item[(j)]
$\bu_{n, 0}$ does not have a subterm of the form $\bt_1\bt_2+\bt_2\bt_3+\bt_3\bt_1$.
\end{itemize}
\end{lem}
\begin{proof}
All statements follow directly from the explicit form of $\bu_{n,k}$;
their verification is straightforward and is omitted.
\end{proof}

\begin{pro}\label{pro26012501}
Let $n \geq 1$ be an integer and $0 \leq k \leq n$.
Then $\mathbf{u}_{m,0}$ is $\mathbf{u}_{n,k}$-free for every $m \geq 4$.
In particular, when $n = m$ we require $k \geq 1$.
\end{pro}
\begin{proof}
Suppose for contradiction that $\bu_{m, 0}$ is not $\bu_{n, k}$-free for some $m \geq 4$, $n\geq1$ and $0\leq k\leq n$.
Then there exist terms $\bp, \br\in P_f(X_c^+)$ and a substitution $\varphi$ such that
\begin{equation}\label{26011601}
\bp\varphi(\bu_{n,k})+\br=\bu_{m,0},
\end{equation}
where $\bp$ may be the empty word, and $\br$ may be the empty set.

If $n>m$, then $n\geq 5$, and so every word in $\bp\varphi(x_1x_2\cdots x_n)$ has length at least $5$.
By Lemma~\ref{lem26011801}(g), $\bp\varphi(x_1x_2\cdots x_n)=x_1x_2\cdots x_m$.
Thus $n\leq m$, contradicting $n>m$.

If $n=m$, then $k \geq 1$ and $n\geq4$, and so $\bp\varphi(x_1x_2\cdots x_n)=x_1x_2\cdots x_m$.
This forces each $\varphi(x_i)$ to be a single variable and $\bp$ is the empty word.
So we may assume that $\varphi(x_j)=x_{i_j}$ for all $1\leq j \leq n$.
By \eqref{26011601},
$\varphi(x_1x_{n+1})$ is an additive subterm of $x_1x_2\cdots x_m+x_{i_1}y_{i_1}x_{m+1}$.
This implies that $\varphi(x_{n+1})$ is an additive subterm of $x_{i_2}\cdots x_{i_m}+y_{i_1}x_{m+1}$,
and so $\varphi(x_2x_{n+1})$ is an additive subterm of $x_{i_2}x_{i_2}\cdots x_{i_m}+x_{i_2}y_{i_1}x_{m+1}$,
which contradicts \eqref{26011601}.

Now suppose that $n<m$.
Our first goal is to extract enough information about $\bp$, $\varphi$, $n$ and $k$ from \eqref{26011601}.
\begin{claim}\label{claim01}
$n\geq2$.
\end{claim}
\begin{proof}[Proof of Claim $\ref{claim01}$.]
If $n=1$, then $k=0$ or $1$, and so $\bu_{n,k}$ is either $x_1+x_1y_1x_2$ or $x_1+x_1x_2$.
Consequently,
\[
\bu_{m,0}=\bp\varphi(\bu_{n,k})+\br=\bp\varphi(x_1+x_1y_1x_2)+\br=\bp\varphi(x_1)+\bp\varphi(x_1)\varphi(y_1x_2)+\br
\]
or
\[
\bu_{m,0}=\bp\varphi(\bu_{n,k})+\br=\bp\varphi(x_1+x_1x_2)+\br=\bp\varphi(x_1)+\bp\varphi(x_1)\varphi(x_2)+\br.
\]
In either case, $\bu_{m,0}$ would contain two distinct words in which one is a proper subword of the other.
This contradicts Lemma~\ref{lem26011801}(e).
Thus $n\geq2$.
\end{proof}

\begin{claim}\label{claim555}
The sets $c(\varphi(x_1)), c(\varphi(x_2)), \ldots, c(\varphi(x_n))$ are pairwise disjoint.
\end{claim}
\begin{proof}
Assume that for some distinct indices $i, j\in\{1, 2, \ldots, n\}$
the intersection $c(\varphi(x_i))\cap c(\varphi(x_j))$ is nonempty.
Then $\varphi(x_1x_2\cdots x_n)$ would have a subterm that is the square of a variable,
and subsequently $\bu_{m,0}$ would also have such a subterm,
contradicting Lemma~\ref{lem26011801}(d).
\end{proof}

\begin{claim}\label{claim0111}
The term $\bp$ is a word.
\end{claim}
\begin{proof}[Proof of Claim $\ref{claim0111}$.]
Suppose that $\bp_1$ and $\bp_2$ are distinct words in $\bp$.
By \eqref{26011601}, both $\bp_1\varphi(x_1x_2\cdots x_n)$ and $\bp_2\varphi(x_1x_2\cdots x_n)$ are contained in $\bu_{m,0}$.
Claim~\ref{claim01} gives $n\geq2$.
Consequently, $\bu_{m,0}$ would contain two distinct words sharing a common subword of length $2$,
which contradicts Lemma~\ref{lem26011801}$(c)$.
\end{proof}

\begin{claim}\label{claim02}
$\displaystyle \bp\varphi\left(\left(\sum_{i=1}^kx_i
               +\sum_{i=k+1}^nx_iy_i\right)x_{n+1}\right)\neq x_1x_2\cdots x_m$.
\end{claim}
\begin{proof}[Proof of Claim $\ref{claim02}$.]
Suppose for contradiction that
\[
\bp\varphi\left(\left(\sum_{i=1}^kx_i
               +\sum_{i=k+1}^nx_iy_i\right)x_{n+1}\right)= x_1x_2\cdots x_m.
\]
If $\bp$ is nonempty, then we may assume that
\[
\bp=x_1x_2\cdots x_t
\]
and
\[
\varphi\left(\left(\sum_{i=1}^kx_i
               +\sum_{i=k+1}^nx_iy_i\right)x_{n+1}\right)=x_{t+1}\cdots x_m
\]
for some $1\leq t<m$.
This implies that $x_j\leq\varphi(x_{n+1})$ for some $t+1\leq j\leq m$,
and that $c(\bp\varphi(x_1\cdots x_n))\subseteq\{x_1, x_2, \dots, x_m\}$.
By \eqref{26011601} and Lemma~\ref{lem26011801}$(f)$, one can deduce that $\bp\varphi(x_1\cdots x_n)=x_1x_2\cdots x_m$,
whence $\varphi(x_1\cdots x_n)=x_{t+1}\cdots x_m$.
Thus $x_j\leq\varphi(x_r)$ for some $1\leq r\leq n$,
and so $x_j^2\leq\varphi(x_rx_{n+1})\leq \bu_{m, 0}$,
contradicting Lemma~\ref{lem26011801}(d).

If $\bp$ is empty, then a similar argument leads to the same contradiction.
\end{proof}

\begin{claim}\label{claim02111}
The word $\bp$ is either empty or a single variable.
\end{claim}
\begin{proof}[Proof of Claim $\ref{claim02111}$.]
By Claim~\ref{claim0111}, $\bp$ is word.
If its length exceeds $1$, then by \eqref{26011601} and Lemma~\ref{lem26011801}(g),
\[
\bp\varphi\left(\left(\sum_{i=1}^kx_i+\sum_{i=k+1}^nx_iy_i\right)x_{n+1}\right)=x_1x_2\cdots x_m.
\]
This contradicts Claim~\ref{claim02}.
Thus the length of $\bp$ is at most $1$;
therefore $\bp$ is either empty or a single variable.
\end{proof}


\begin{claim}\label{claim021122}
The word $\bp$ is empty.
\end{claim}
\begin{proof}[Proof of Claim $\ref{claim021122}$.]
Suppose for contradiction that $\bp$ is nonempty.
By Claim~\ref{claim02111}, $\bp$ is then a single variable.
Claim~\ref{claim01} tells us that $n\geq 2$.

Observe that every word in $\bp\varphi(x_iy_ix_{n+1})$ has length at least $4$ for every $k+1\leq i\leq n$.
By Lemma~\ref{lem26011801}(g),
\begin{equation}\label{26011801}
\bp\varphi(x_iy_ix_{n+1})=x_1x_2\cdots x_m \quad (k+1\leq i\leq n).
\end{equation}
Combining Claim~\ref{claim02}, one can deduce that $k\geq 1$.

If $k\geq 2$, then by Lemma~\ref{lem26011801}(c),
$\bp\varphi(x_1x_{n+1})=\bp\varphi(x_2x_{n+1})$.
Hence $\varphi(x_1)=\varphi(x_2)$, and so the square $\varphi(x_1)\varphi(x_1)$
is a subterm of $\bu_{m, 0}$, which contradicts Lemma~\ref{lem26011801}(d).

If $k=1$, then by \eqref{26011801}, $\bp\varphi(x_iy_ix_{n+1})=x_1x_2\cdots x_m$ for every $2\leq i\leq n$.
So $\bp=x_j$ for some $1\leq j \leq m$.
Since $\bp\varphi(x_1x_{n+1})$ and $\bp\varphi(x_2y_2x_{n+1})$ share a common subword of length $2$,
it follows from Lemma~\ref{lem26011801}(c) that $\bp\varphi(x_1x_{n+1})=\bp\varphi(x_2y_2x_{n+1})$.
Consequently,
\[
\bp\varphi\left(\left(x_1+\sum_{i=2}^nx_iy_i\right)x_{n+1}\right)= x_1x_2\cdots x_m,
\]
which contradicts Claim~\ref{claim02}.
\end{proof}

Now we complete the main proof of the proposition.
By Claim~\ref{claim021122}, the equality~\eqref{26011601} reduces to the form
\begin{equation}\label{26011810}
\varphi(\bu_{n,k})+\br=\bu_{m,0}.
\end{equation}
Hence, in the subset representation, $\varphi(\bu_{n,k})$ is a subset of $\bu_{m,0}$.
From Claim~\ref{claim02} we obtain that
\begin{equation}\label{226011950}
\bt=\varphi\left(\left(\sum_{i=1}^kx_i
               +\sum_{i=k+1}^nx_iy_i\right)x_{n+1}\right)
               =\varphi\left(\sum_{i=1}^kx_i
               +\sum_{i=k+1}^nx_iy_i\right)\varphi(x_{n+1})
\end{equation}
contains a word of the form $x_iy_ix_{m+1}$ for some $1\leq i \leq m$.

\textbf{Case 1.} The term $\bt$ contains at least two distinct words.
Then $\bt$ must also contain either $x_1x_2\cdots x_m$ or $x_jy_jx_{m+1}$ for some $j\neq i$ ($1\leq j\leq m$).

\textbf{Subcase 1.1.} The term $\bt$ contains $x_1x_2\cdots x_m$. Then
\[
\bt=x_1x_2\cdots x_m+x_iy_ix_{m+1}=x_i(x_1x_2\cdots x_{i-1}x_{i+1}\cdots x_m+y_ix_{m+1}).
\]
If this were impossible, $\bt$ would has no nontrivial multiplicative subterm, contradicting \eqref{226011950}.

By Claim~\ref{claim555}, $\varphi(x_{n+1})=x_i$ and
\[
\varphi\left(\sum_{i=1}^kx_i+\sum_{i=k+1}^nx_iy_i\right)=x_1x_2\cdots x_{i-1}x_{i+1}\cdots x_m+y_ix_{m+1}.
\]
Consequently,
\[
c(\varphi(x_1x_2\cdots x_n))\subseteq \{x_1,x_2,\dots, x_{i-1}, y_i, x_{i+1},\dots, x_m, x_{m+1}\},
\]
so $x_i, y_j\notin c(\varphi(x_1x_2\cdots x_n))$ for every $j\neq i$ ($1\leq j \leq m$).
Hence every word in $\varphi(x_1x_2\cdots x_n)$ does not belong to $\bu_{m,0}$,
and so $\varphi(x_1x_2\cdots x_n)\nsubseteq \bu_{m,0}$, contradicting \eqref{26011810}.

\textbf{Subcase 1.2.}
The term $\bt$ contains $x_jy_jx_{m+1}$ for some $j\neq i$ ($1\leq j \leq m$).
Then $\bt$ does not contain $x_1x_2\cdots x_m$.
By \eqref{226011950},
we may write
\[
\bt=\sum_{i=1}^\ell x_iy_ix_{m+1}
\]
for some $2\leq\ell\leq m$. Since $\bt$ has the unique multiplicative factorization
\[
\bt=\left(\sum_{i=1}^\ell x_iy_i\right)x_{m+1},
\]
it follows from Claim~\ref{claim555} that $\varphi(x_{n+1})=x_{m+1}$ and
\begin{equation}\label{26012901}
\varphi\left(\sum_{i=1}^kx_i+\sum_{i=k+1}^nx_iy_i\right)=\sum_{{i=1}}^\ell x_iy_i.
\end{equation}
Thus
\begin{equation}\label{26012902}
c(\varphi(x_1x_2\cdots x_n))\subseteq\{x_1, y_1,\dots, x_\ell, y_\ell\},
\end{equation}
and so $x_{m+1}\notin c(\varphi(x_1x_2\cdots x_n))$.
From \eqref{26011810} we have that $\varphi(x_1x_2\cdots x_n)=x_1x_2\cdots x_m$.
Combining \eqref{26012901} and \eqref{26012902},
one can deduce that $\ell=m$, $k=0$ and so
\[
\varphi\left(\sum_{i=1}^nx_iy_i\right)=\sum_{{i=1}}^m x_iy_i.
\]
Therefore, $\varphi(x_i)$ is a variable for every $1\leq i\leq n$, which forces $n=m$, a contradiction.

\textbf{Case 2.} The term $\bt$ coincides with $x_iy_ix_{m+1}$. By Claim~\ref{claim555}, $\varphi(x_{n+1})$ is a single variable;
we may assume that $\varphi(x_{n+1})=x_i$.
Then
\[
\varphi\left(\sum_{1\leq i\leq k}x_i+\sum_{k+1\leq i\leq n}x_iy_i\right)=y_ix_{m+1}.
\]
It follows that $c(\varphi(x_1x_2\cdots x_n))\subseteq\{y_i, x_{m+1}\}$,
and so $\varphi(x_1x_2\cdots x_n)\nsubseteq \bu_{m,0}$, contracting \eqref{26011810}.

This completes the proof.
\end{proof}

Let $n\geq 1$ be an integer. Define
\[
\bq_n=\prod_{i=1}^{n+1}x_i.
\]
For any $0\leq k \leq n$,
we denote by $\sigma_{n, k}$ the inequality $\bq_n \preceq \bu_{n, k}$,
and write $\Omega_0$ for the set of all inequalities $\sigma_{n, 0}$ with $n\geq 1$.

\begin{thm}\label{thm26012401}
Let $S$ be a commutative ai-semiring
and $\Sigma$ an equational basis for $S$ that contains an infinite subset of $\Omega_0$.
If $\bu_{m, 0}$ is $\bt$-free for every inequality $\bs\preceq\bt$ in $\Sigma\setminus \Omega_0$
and for every $m\geq 4$, then $S$ is nonfinitely based.
\end{thm}

\begin{proof}
By hypothesis and the compactness theorem of equational logic,
it is enough to prove that no finite subset of the set $\Sigma$ defines the variety $\mathsf{V}(S)$.

Let $\Sigma'$ be an arbitrary finite subset of $\Sigma$.
Since $\Sigma$ contains an infinite subset of $\Omega_0$, we can choose an integer $m$ such that
\[
m>\max\{n\geq 1 \mid \sigma_{n, 0}\in \Sigma'\}~ \text{and} ~\sigma_{m, 0}\in\Sigma.
\]
Then $\sigma_{m, 0}\notin\Sigma'$.

To prove that $\Sigma'$ can not define the variety $\mathsf{V}(S)$,
it suffices to verify that $\sigma_{m, 0}$ is not derivable from $\Sigma'$.
By hypothesis, together with Lemma~\ref{lem26012701} and Lemma~\ref{lem02},
it is sufficient to show that $\bu_{m, 0}$ is $\bu_{n, 0}$-free for all $1 \leq n<m$,
which follows from Proposition~\ref{pro26012501} immediately.
This completes the proof.
\end{proof}

In the remainder we apply Theorem~\ref{thm26012401} to show that $S_{(4, 545)}$ is nonfinitely based.
To this end, we first establish the following equality of varieties,
which shows that $\mathsf{V}(S_{(4, 545)})$ is the join of two finitely based varieties
$\mathsf{V}(S_{53})$ and $\mathsf{V}(D_2)$.
(For a class $\mathcal{K}$ of ai-semirings,
$\mathsf{V}(\mathcal{K})$ denotes the variety generated by $\mathcal{K}$.)

\begin{pro}\label{pro26011401}
$\mathsf{V}(S_{(4, 545)})=\mathsf{V}(S_{53}, D_2)$.
\end{pro}
\begin{proof}
Let $M_2$ denote the $2$-element ai-semiring with Cayley tables shown in Table~\ref{tbM_2},
and let $S_{43}$ denote the $3$-element ai-semiring given in Table~\ref{tb43}.
It is straightforward to check that $S_{(4, 545)}$ is isomorphic to a subdirect product of $S_{43}$ and $S_{53}$
via the congruences defined by the nontrivial blocks $\{\infty, a\}$ and $\{1, 0\}$.
Hence
\[
\mathsf{V}(S_{(4, 545)}) = \mathsf{V}(S_{53}, S_{43}).
\]

Both $M_2$ and $D_2$ embed into $S_{43}$.
Moreover, $S_{43}$ satisfies the following equational basis of $\mathsf{V}(M_2, D_2)$ (see~\cite{sr}):
\[
xy \approx yx, x^2 \approx x, x+yz \approx x+yz+xz+yx.
\]
Thus $\mathsf{V}(S_{43})=\mathsf{V}(M_2, D_2)$,
and so $\mathsf{V}(S_{(4, 545)})=\mathsf{V}(S_{53}, M_2, D_2)$.

Since $M_2$ embeds into $S_{53}$,
we finally obtain
\[
\mathsf{V}(S_{(4, 545)})=\mathsf{V}(S_{53}, D_2).\qedhere
\]
\end{proof}

\begin{cor}\label{coro26012301}
$\mathsf{V}(S_{(4, 545)})$ is a subvariety of $\mathsf{V}(S_{(4, 634)})$.
\end{cor}
\begin{proof}
Since $S_{(4, 634)}$ is isomorphic to $S^0_{53}$, and $S^0_{53}$ contains $S_{53}$ and a copy of $D_2$,
Proposition~\ref{pro26011401} immediately yields the required result.
\end{proof}

\begin{table}[ht]
\caption{The Cayley tables of $M_2$} \label{tbM_2}
\begin{tabular}{c|cc}
$+$      &$0$&$1$\\
\hline
$0$      &$0$&$1$\\
$1$      &$1$&$1$\\
\end{tabular}\qquad
\begin{tabular}{c|cc}
$\cdot$      &$0$&$1$\\
\hline
$0$      &$0$&$1$\\
$1$      &$1$&$1$\\
\end{tabular}
\end{table}

\begin{table}[ht]
\caption{The Cayley tables of $S_{43}$} \label{tb43}
\begin{tabular}{c|ccc}
$+$      &$\infty$&$a$&$1$\\
\hline
$\infty$ &$\infty$&$\infty$&$\infty$\\
$a$      &$\infty$&$a$&$a$\\
$1$      &$\infty$&$a$&$1$\\
\end{tabular}\qquad
\begin{tabular}{c|ccc}
$\cdot$  &$\infty$&$a$&$1$\\
\hline
$\infty$ &$\infty$&$\infty$&$\infty$\\
$a$      &$\infty$&$a$&$1$\\
$1$      &$\infty$&$1$&$1$\\
\end{tabular}
\end{table}

The following result completely describes the identities of $S_{53}$.
\begin{lem}\label{lem5303}
Let $\bq\preceq\bu$ be a nontrivial inequality such that
$\bu=\bu_1+\bu_2+\cdots+\bu_n$ and $\bu_i, \bq \in X^+_c$ for all $1\leq i \leq n$.
Then $\bq\preceq\bu$ holds in $S_{53}$
if and only if $L_{\geq 2}(\bu)\neq \emptyset$, $c(\bq)\subseteq c(\bu)$,
and for every $\bw\in S_2(\bq)$ there exists $\bw'\in S_2(\bu)$ such that $c(\bw')\subseteq c(\bw)$.
\end{lem}
\begin{proof}
The necessity follows from Yue et al.~\cite[Lemma 3.7]{yrzs}.
For sufficiency,
let $\varphi\colon P_f(X) \to S_{53}$ be an arbitrary homomorphism.

If $\varphi(\bu)=\infty$, then $\varphi(\bq)\leq \varphi(\bu)$,
since $\infty$ is the additive top element.

If $\varphi(\bu)=1$, then $\varphi(x)=1$ for all $x\in c(\bu)$,
since $1$ is the additive least element and its only multiplicative factor is itself.
As $c(\bq)\subseteq c(\bu)$, the same holds for all variables in $c(\bq)$;
hence $\varphi(\bq)=1$ and so $\varphi(\bq)\leq \varphi(\bu)$.

If $\varphi(\bu)=a$, then $\varphi(x)\in \{1, a\}$ for all $x\in c(\bu)$.
Indeed, if there is $x\in c(\bu)$ such that $\varphi(x)=\infty$,
then because $\infty$ is the multiplicative zero and the additive top element,
we would obtain $\varphi(\bu)=\infty$, contradicting the assumption $\varphi(\bu)=a$.
Since $c(\bq)\subseteq c(\bu)$, the same holds for all variables in $c(\bq)$.
Suppose for contradiction that $\varphi(\bq)=\infty$.
Then there exists $xy\in S_2(\bq)$ such that $\varphi(x)=\varphi(y)=a$.
(Here the reasoning is that in $S_{53}$, the product of finitely many elements,
each equal to $a$ or $1$, yields $\infty$ precisely when at least two factors equal to $a$.)
By hypothesis, there exists $\bp\in S_2(\bu)$ such that $c(\bp)\subseteq \{x, y\}$,
and so $\varphi(\bp)=aa=\infty$.
This implies that $\varphi(\bu)=\infty$, a contradiction.
Thus $\varphi(\bq)\in \{1, a\}$ and so $\varphi(\bq)\leq \varphi(\bu)$.

Therefore, $\bq\preceq\bu$ holds in $S_{53}$.
\end{proof}

We now present an infinite equational basis for $S_{(4, 545)}$.
\begin{pro}\label{pro54501}
$\mathsf{V}(S_{(4, 545)})$ is the commutative ai-semiring variety defined by the identities
\begin{align}
&x^3 \approx x^2; \label{54501}\\
&x \preceq  x^2; \label{54503}\\
&xy \preceq x+y^2z;\label{54504}\\
&\sigma_{n,k} \quad (n \geq 1, ~0\leq k\leq n).\label{54505}
\end{align}
\end{pro}
\begin{proof}
It is a routine matter to verify that $S_{(4, 545)}$ satisfies the identities \eqref{54501}--\eqref{54504}.
By Proposition~\ref{pro26011401} together with Lemmas~\ref{lem01} and \ref{lem5303},
$S_{(4, 545)}$ satisfies the inequalities in \eqref{54505}.
It remains to show that every inequality that holds in $S_{(4, 545)}$
is derivable from \eqref{54501}--\eqref{54505}.
Let $\bq\preceq \bu$ be such a nontrivial inequality, where
$\bu=\bu_1+\bu_2+\cdots+\bu_n$ and $\bu_i, \bq \in X_c^+$, $1 \leq i \leq n$.
By Lemmas \ref{lem01} and \ref{lem5303},
$c(\bu_i)\subseteq c(\bq)$ for some $\bu_i\in\bu$,
$L_{\geq 2}(\bu)\neq \emptyset$, $c(\bq)\subseteq c(\bu)$,
and for any $\bw\in S_2(\bq)$, there exists $\bw'\in S_2(\bu)$ such that $c(\bw')\subseteq c(\bw)$.

\textbf{Case 1.} $\ell(\bq)=1$. Then $\bq=x$ for some $x\in X$. So there exists $\bu_i\in \bu$ such that $\bu_i=x^k$ for some $k\geq2$.
Now we have
\[
\bu \succeq \bu_i\approx x^k\stackrel{\eqref{54501}}\approx x^2\stackrel{\eqref{54503}}\succeq x=\bq.
\]
This derives the inequality $\bu \succeq \bq$.

\textbf{Case 2.} $\ell(\bq)\geq2$.
Then $\bq=x_1x_2\cdots x_n$ for some $n\geq 2$ and $x_1, x_2, \ldots, x_n \in X$.
For any $x_ix_j\in S_2(\bq)$, it follows from Lemma \ref{lem5303} that
there exists $\bw_{ij}\in S_2(\bu)$ such that $c(\bw_{ij})\subseteq \{x_i, x_j\}$, and so $\bw_{ij}=x_i^2$ or $x_j^2$ or $x_ix_j$.
If $\bw_{ij}=x_i^2$, then there exists $\bu_{ij} \in \bu$ such that
$\bu_{ij}=x_i^2\bu_{ij}'$ for some $\bu_{ij}'\in X^*$.
Since $c(\bq)\subseteq c(\bu)$, we have that there exists $\bu_\ell \in \bu$ such that $\bu_\ell=x_j\bu_\ell'$ for some $\bu_\ell'\in X^*$,
and so
\[
\bu \succeq \bu_\ell+\bu_{ij}=x_j\bu_\ell'+x_i^2\bu_{ij}'
\stackrel{\eqref{54504}}\succeq x_j\bu_\ell'x_i=x_ix_j\bu_\ell'.
\]
This implies the inequality $\bu\succeq x_ix_j\bu_\ell'$.
If $\bw_{ij}=x_j^2$, then we can derive $\bu \succeq x_ix_j\bu_k'$ for some $\bu_k'\in X^*$ by a similar way.
If $\bw_{ij}=x_ix_j$, then there exists $\bu_r\in \bu$ such that $\bu_r=x_ix_j\bu_r'$ for some $\bu_r'\in X^*$,
and so $\bu \succeq x_ix_j\bu_r'$ is derived.
Consequently, for any $x_ix_j\in S_2(\bq)$, one can always obtain the inequality
$\bu\succeq x_ix_j\bp_{ij}$ for some $\bp_{ij}\in X^*$.
Since $c(\bu_i)\subseteq c(\bq)$, we may assume that $\bu_i=x_1x_2\cdots x_t$ for some $t<n$.
By the preceding argument, for any $1\leq s\leq t$,
$x_sx_{t+1}\bp_{s, t+1}\in\bu$ for some $\bp_{s, t+1}\in X^*$.
Now we have
\[
\bu\succeq x_1x_2\cdots x_t+x_1x_{t+1}\bp_{_{1,t+1}}+\cdots+x_tx_{t+1}\bp_{_{t,t+1}}\stackrel{\eqref{54505}}\succeq x_1x_2\cdots x_{t+1}.
\]
This implies the inequality $\bu\succeq x_1x_2\cdots x_{t+1}$.
Repeat this process one can finally derive the inequality $\bu \succeq x_1x_2\cdots x_n=\bq$.
\end{proof}

\begin{cor}\label{pro01}
$\mathsf{V}(S_{(4, 545)})$ contains $S_7$.
\end{cor}
\begin{proof}
By Proposition~\ref{pro54501},
it is enough to prove that $S_7$ satisfies the identities \eqref{54501}--\eqref{54505}.
One can directly check that the identities \eqref{54501}--\eqref{54504} hold in $S_7$;
we therefore focus on the the inequalities in \eqref{54505}.

From \cite[Proposition 5.5]{gmrz},
an inequality $\bq\preceq\bu$ holds in $S_7$
if and only if $c(\bq)\subseteq c(\bu)$ and $\delta(\bu)\subseteq\delta(\bu+\bq)$.
Let $n \geq 1$, $0\leq k\leq n$.
It is evident that $c(\bq_n)\subseteq c(\bu_{n,k})$.
The inclusion $\delta(\bu_{n,k})\subseteq\delta(\bu_{n,k}+\bq_n)$ is checked by considering three cases:
\begin{itemize}
\item If $k\geq 2$, then $\delta(\bu_{n,k})=\emptyset$.

\item If $k=1$, then
$
\delta(\bu_{n,k})=\{\{x_1, y_{2}, \ldots, y_n\}\}.
$

\item If $k=0$, then
\[
\delta(\bu_{n,k})=\{\{y_1, \ldots, y_{i-1}, x_i, y_{i+1},\ldots, y_n\}\mid 1\leq i \leq n\}.
\]
\end{itemize}
In every case we have that $\delta(\bu_{n,k})\subseteq\delta(\bu_{n,k}+\bq_n)$.
Thus the inequality \eqref{54505} is satisfied by $S_7$.
\end{proof}

\begin{cor}\label{coro01}
$\mathsf{V}(S_{(4, 545)})$ has $2^{\aleph_0}$ distinct subvarieties.
\end{cor}
\begin{proof}
This follows from Corollary~\ref{pro01} together with one of the main results of \cite{gmrz},
which states that $\mathsf{V}(S_7)$ has $2^{\aleph_0}$ distinct subvarieties.
\end{proof}

\begin{thm}\label{thm545}
The ai-semiring $S_{(4, 545)}$ is nonfinitely based.
\end{thm}
\begin{proof}
By Proposition~\ref{pro54501},
\[
\Sigma=\{\eqref{54501}, \eqref{54503}, \eqref{54504}\}\cup\{\sigma_{n, k} \mid n\geq 1, 0\leq k\leq n\}
\]
is an equational basis for the commutative ai-semiring $S_{(4, 545)}$, which contains $\Omega_0$.
By Theorem~\ref{thm26012401}, to prove that $S_{(4, 545)}$ is nonfinitely based,
it suffices to show that for any $m\geq 4$, $\bu_{m, 0}$ is $\bw$-free for every term $\bw$ in the set
\[\{x^3, x^2, x+y^2z\}\cup\{\bu_{n, k}\mid n\geq 1, 1\leq k\leq n\}.
\]
Indeed, the required freeness follows from Lemma~\ref{lem26011801}(d) for $\mathbf{w}\in\{x^3,x^2,x+y^2z\}$,
and from Proposition~\ref{pro26012501} for $\mathbf{w}=\mathbf{u}_{n,k}$ ($n\geq1$, $1\leq k\leq n$).
\end{proof}

\section{Equational basis for $S_{(4, 634)}$}\label{sec5}
In this section, we give another sufficient condition for an ai-semiring to be nonfinitely based.
Applying this condition, we prove that $S_{(4, 634)}$ likewise has no finite equational basis.

We continue to use the notation of Section~\ref{sec4}.
In particular, for each integer $n \geq 1$,
\[
\mathbf{u}_{n,n}=\prod_{i=1}^n x_i+\sum_{i=1}^n x_ix_{n+1}.
\]

For each integer $n\geq 2$, let $\Theta_n$ denote the set of all terms of the form
\[
\sum_{1\leq i<j\leq n+1}x_ix_j\bw_{ij},
\]
where $c(\bw_{ij})\subseteq\{x_1, x_2,\dots,x_{n+1}\}\backslash\{x_i, x_j\}$
and each $\bw_{ij}$ is either a linear word or the empty word.
For every term $\bt\in\Theta_n$, each word in $\bt$ is linear and has length at least $2$,
and for any $1\leq i<j\leq n+1$, $x_ix_j$ is a subterm of $\bt$.
Observe that $\bu_{n, n}$ lies in $\Theta_n$.

\begin{pro}\label{pro26012502}
Let $n\geq 2$ and $m\geq 3$ be integers.
If $\bv$ is a term in $\Theta_n$ different from $\bu_{m, m}$, then $\bu_{m,m}$ is $\bv$-free.
In particular, if $m\neq n$, then $\bu_{m,m}$ is $\bu_{n, n}$-free.
\end{pro}

\begin{proof}
Suppose for contradiction that $\bu_{m, m}$ is not $\bv$-free for some
$m\geq 3$, $n \geq 2$ and $\bv\in \Theta_n \setminus \{\bu_{m,m}\}$.
Then there exist terms $\bp, \br\in P_f(X_c^+)$ and a substitution $\varphi$ such that
\begin{equation}\label{26012001}
\bp\varphi(\bv)+\br=\bu_{m,m},
\end{equation}
where $\bp$ may be the empty word, and $\br$ may be the empty set.

\begin{claim}\label{claim63401}
For distinct words $x_{i_1}x_{j_1}\bw_{i_1j_1}$ and $x_{i_2}x_{j_2}\bw_{i_2j_2}$ in $\bv$,
\[
\bp\varphi(x_{i_1}x_{j_1}\bw_{i_1j_1})\neq\bp\varphi(x_{i_2}x_{j_2}\bw_{i_2j_2}).
\]
\end{claim}
\begin{proof}[Proof of Claim $\ref{claim63401}$.]
By the equality~\eqref{26012001} and Lemma~\ref{lem26011801}(e),
the sets $c(x_{i_1}x_{j_1}\bw_{i_1j_1})$ and $c(x_{i_2}x_{j_2}\bw_{i_2j_2})$ are incomparable.
Suppose for contradiction that
\[
\bp\varphi(x_{i_1}x_{j_1}\bw_{i_1j_1})=\bp\varphi(x_{i_2}x_{j_2}\bw_{i_2j_2}).
\]
Then
\begin{equation}\label{26012301}
\varphi(x_{i_1}x_{j_1}\bw_{i_1j_1})=\varphi(x_{i_2}x_{j_2}\bw_{i_2j_2}).
\end{equation}

\textbf{Case 1.} $c(x_{i_1}x_{j_1}\bw_{i_1j_1})\cap c(x_{i_2}x_{j_2}\bw_{i_2j_2})$ is empty.
Take a variable $y$ in $\varphi(x_{i_1})$. By \eqref{26012301}, $y$ must appear in
$\varphi(x_{i_2})$, $\varphi(x_{j_2})$, or $\varphi(x_s)$ for some $x_s\leq \bw_{i_2j_2}$.
Consequently, $y^2$ is a subterm of
$\varphi(x_{i_1}x_{i_2})$, $\varphi(x_{i_1}x_{j_2})$, or $\varphi(x_{i_1}x_s)$.
Therefore, $y^2\leq \bu_{m, m}$, which contradicts Lemma~\ref{lem26011801}(d).

\textbf{Case 2.} $c(x_{i_1}x_{j_1}\bw_{i_1j_1})\cap c(x_{i_2}x_{j_2}\bw_{i_2j_2})$ is nonempty.
Without loss of generality, we may assume that $x_{i_1}=x_{i_2}$.
By \eqref{26012301}, it follows that
\[
\varphi(x_{j_1}\bw_{i_1j_1})=\varphi(x_{j_2}\bw_{i_2j_2}).
\]

If $c(x_{j_1}\bw_{i_1j_1})\cap c(x_{j_2}\bw_{i_2j_2})$ is empty,
we are back to Case~1 and obtain a contradiction.
If the intersection remains nonempty,
we iterate the same step: after finitely many iterations
we must arrive at two words with disjoint contents whose $\varphi$-images coincide, which again reduces to Case~1.
\end{proof}

\begin{claim}\label{claim63402}
The term $\bp$ is empty.
\end{claim}
\begin{proof}[Proof of Claim $\ref{claim63402}$.]
Suppose for contradiction that $\bp$ is nonempty.
Then by Lemma~\ref{lem26011801}(h),
\[
\bp\varphi(x_{i_1}x_{j_1}\bw_{i_1j_1})=\bp\varphi(x_{i_2}x_{j_2}\bw_{i_2j_2})=x_1x_2\dots x_m
\]
for any distinct terms $x_{i_1}x_{j_1}\bw_{i_1j_1}$ and $x_{i_2}x_{j_2}\bw_{i_2j_2}$ in $\bv$,
contradicting Claim~\ref{claim63401}.
\end{proof}

By Claim~\ref{claim63402}, the equality~\eqref{26012001} reduces to the form
\begin{equation}\label{26012002}
\varphi(\bv)+\br=\bu_{m,m}.
\end{equation}
Hence, in the subset representation, $\varphi(\bv)$ is a subset of $\bu_{m,m}$.

\begin{claim}\label{claim63403}
The term $\bv$ contains exactly one word of length at least $3$.
\end{claim}
\begin{proof}[Proof of Claim $\ref{claim63403}$.]
We argue by contradiction. First, suppose that $\bv$ contains two distinct words of length at least $3$,
say $x_{i_1}x_{j_1}\bw_{i_1j_1}$ and $x_{i_2}x_{j_2}\bw_{i_2j_2}$.
Then by Lemma~\ref{lem26011801}(h), we would have
\[
\varphi(x_{i_1}x_{j_1}\bw_{i_1j_1})=\varphi(x_{i_2}x_{j_2}\bw_{i_2j_2})=x_1x_2\dots x_m,
\]
contradicting Claim~\ref{claim63401}.

Now suppose, to the contrary, that every word in $\bv$ has length $2$.
This forces $\bw_{ij}=\{1\}$ for all $1\leq i<j\leq n+1$, and consequently
\[
\varphi(\bv)=\varphi\left(\sum_{1\leq i<j\leq n+1}x_ix_j\right).
\]
It follows that
\[
\varphi(x_1)\varphi(x_2)+\varphi(x_1)\varphi(x_3)+\varphi(x_2)\varphi(x_3)\subseteq \bu_{m,m},
\]
which contradicts Lemma~\ref{lem26011801}(i).
\end{proof}

Now we complete the main proof of the proposition.
By Lemma~\ref{lem26011801}(e), \eqref{26012002} and Claim~\ref{claim63403}, we may write
\begin{equation}\label{26012601}
\bv=
\prod_{i=1}^k x_i+\sum_{1\leq i\leq k, k+1\leq j\leq n+1}x_ix_j+\sum_{k+1\leq i<j\leq n+1}x_ix_j,
\end{equation}
where $3\leq k\leq n$.
Combining \eqref{26012002} and Lemma~\ref{lem26011801}(h), we have
\begin{equation}\label{26012503}
\varphi(x_1x_2\cdots x_k)=x_1x_2\dots x_m.
\end{equation}
This implies that $k\leq m$ and $\varphi(x_i)$ is a word for all $1\leq i\leq k$.

If $k<m$, then by \eqref{26012503}, there exists $1\leq i\leq k$ such that the word $\varphi(x_i)$ has length at least $2$.
Choose $j$ such that $k+1\leq j\leq n+1$.
Then $x_ix_j$ is a word in $\bv$.
By Lemma~\ref{lem26011801}(h) and \eqref{26012503},
\[
\varphi(x_ix_j)=x_1x_2\dots x_m=\varphi(x_1x_2\cdots x_k),
\]
which contradicts Claim~\ref{claim63401}.

If $k=m$, then by \eqref{26012601}, $k+1<n+1$, since $\bv\neq\bu_{m, m}$.
By \eqref{26012503}, each $\varphi(x_i)$ must be a single variable.
So we may assume that $\varphi(x_j)=x_{i_j}$ for all $1\leq j \leq k$.
By \eqref{26012002},
$\varphi(x_1x_{k+1})$ and $\varphi(x_1x_{k+2})$ are additive subterms of $x_1x_2\cdots x_m+x_{i_1}x_{m+1}$.
This implies that both $\varphi(x_{k+1})$ and $\varphi(x_{k+2})$ are additive subterms of $x_{i_2}\cdots x_{i_m}+x_{m+1}$,
and so $\varphi(x_{k+1}x_{k+2})$ is an additive subterm of $(x_{i_2}\cdots x_{i_m}+x_{m+1})^2$,
which contradicts \eqref{26012002}.

This completes the proof.
\end{proof}

As in Section~\ref{sec4}, we let $\sigma_{n, n}$ denote the inequality $\bq_n \preceq \bu_{n, n}$,
where
\[
\bq_n=\prod_{i=1}^{n+1}x_i.
\]
Write $\Omega$ for the set of all inequalities $\sigma_{n, n}$ with $n\geq 2$.

\begin{thm}\label{thm26012501}
Let $S$ be a commutative ai-semiring
and $\Xi$ an equational basis for $S$ that contains an infinite subset of $\Omega$.
If $\bu_{m, m}$ is $\bt$-free for every inequality $\bs\preceq\bt$ in $\Xi\setminus \Omega$
and for every $m\geq 3$, then $S$ is nonfinitely based.
\end{thm}
\begin{proof}
By hypothesis and the compactness theorem of equational logic,
it is enough to prove that no finite subset of the set $\Xi$ defines the variety $\mathsf{V}(S)$.

Let $\Xi'$ be an arbitrary finite subset of $\Xi$.
Since $\Xi$ contains an infinite subset of $\Omega$, we can choose an integer $m$ such that
\[
m>\max\{n\geq 2\mid  \sigma_{n, n} \in \Xi'\} ~\text{and}~\sigma_{m, m}\in\Xi.
\]
Then $\sigma_{m, m}\notin\Xi'$.

To prove that $\Xi'$ can not define the variety $\mathsf{V}(S)$,
it suffices to verify that $\sigma_{m, m}$ is not derivable from $\Xi'$.
By hypothesis, together with Lemma~\ref{lem26012701} and Lemma~\ref{lem02},
it is sufficient to show that $\bu_{m, m}$ is $\bu_{n, n}$-free for all $2 \leq n<m$,
which follows from Proposition~\ref{pro26012502} immediately.
This completes the proof.
\end{proof}

Recall that $S_{(4, 634)}$ is isomorphic to $S^0_{53}$.
The following result, due to Wu et al.~\cite[Proposition 1.5]{wrz},
explores the relationship between the equational theories of $S^0$ and $S$.

\begin{lem}\label{lem04}
Let $\bq\preceq \bu$ be an ai-semiring identity such that
$\bu=\bu_1+\bu_2+\cdots+\bu_n$, where $\bu_i, \bq \in X^+$, $1\leq i \leq n$.
Then $\bq\preceq \bu$ is satisfied by ${S}^0$ if and only if
$\bq\preceq D_\bq(\bu)$ holds in $S$.
\end{lem}

For $n \geq 2$ and $\bv\in \Theta_n$,
let $\delta_{n,\bv}$ denote the inequality $\bq_n \preceq \bv$.
Then $\sigma_{n,n}$ coincides with $\delta_{n,\bv}$ for some $\bv\in \Theta_n$.
We now present an infinite equational basis for $S_{(4, 634)}$.
\begin{pro}\label{pro63401}
$\mathsf{V}(S_{(4, 634)})$
is the commutative ai-semiring variety defined by the identities \eqref{54501}, \eqref{54503}, and
\begin{align}
&xy \preceq x^2+y;\label{63404}\\
&\delta_{n,\bv} \quad (n \geq 2, \bv\in \Theta_n); \label{63405}\\
&\bq_n \preceq \left(\sum_{i=1}^{n}x_i^2\bw_i\right)+x_{n+1}\bw_{n+1}\quad (n \geq 2); \label{63406}\\
&\bq_n \preceq \left(\sum_{i=1}^{k}x_i^2\bw_i\right)+\left(\sum_{k+1\leq i<j\leq n+1}x_ix_j\bw_{ij}\right)\quad (n \geq 2, 1\leq k<n+1),\label{63407}
\end{align}
where in \eqref{63406} and \eqref{63407} each $\bw_i$ satisfies
$c(\bw_i) \subseteq \{x_1, x_2, \dots, x_{n+1}\} \setminus \{x_i\}$
and is either a linear word or the empty word;
in \eqref{63407} each $\bw_{ij}$ satisfies
$c(\bw_{ij}) \subseteq \{x_1, x_2, \dots, x_{n+1}\} \setminus \{x_i, x_j\}$
and is either a linear word or the empty word.
\end{pro}
\begin{proof}
It is straightforward to verify that $S_{(4, 634)}$ satisfies the identities \eqref{54501}, \eqref{54503}, and \eqref{63404}.
By Lemmas \ref{lem5303} and \ref{lem04}, $S_{(4, 634)}$ also satisfies the inequalities in \eqref{63405}--\eqref{63407}.
It suffices to show that every inequality holding in $S_{(4, 634)}$
is derivable from \eqref{54501}, \eqref{54503}, \eqref{63404}--\eqref{63407}.
Consider such a nontrivial inequality $\bq\preceq \bu$, where
$\bu=\bu_1+\bu_2+\cdots+\bu_n$ and $\bu_i, \bq \in X_c^+$, $1 \leq i \leq n$.
By Lemma~\ref{lem04}, $\bq\preceq D_\bq(\bu)$ holds in $S_{53}$.
Lemma~\ref{lem5303} further implies that
$L_{\geq 2}(D_\bq(\bu))\neq \emptyset$, $c(\bq)=c(D_\bq(\bu))$,
and for any $\bw\in S_2(\bq)$, there exists $\bw'\in S_2(D_\bq(\bu))$ such that $c(\bw')\subseteq c(\bw)$.

\textbf{Case 1.} $\ell(\bq)=1$. Then $\bq=x$ for some $x\in X$. Since $L_{\geq 2}(D_\bq(\bu))\neq \emptyset$,
there exists $\bu_j\in D_\bq(\bu)$ such that $\ell(\bu_j)\geq2$, and so $\bu_j=x^k$ for some $k\geq2$.
Now we have
\[
\bu \succeq\bu_j=x^k\stackrel{\eqref{54501}}\approx x^2\stackrel{\eqref{54503}}\succeq x=\bq,
\]
which derives the inequality $\bu \succeq \bq$.

\textbf{Case 2.} $\ell(\bq)\geq2$.
Then we may write $\bq=x_1x_2\cdots x_{n+1}$ for some $n\geq 1$ and $x_1, x_2, \ldots, x_{n+1} \in X$.
For any $1\leq i<j\leq n+1$, we have that $x_ix_j\in S_2(\bq)$,
and so there exists $\bw_{ij}\in S_2(D_\bq(\bu))$ such that $c(\bw_{ij})\subseteq \{x_i, x_j\}$.
Consequently, $\bw_{ij}$ must be one of $x_i^2$, $x_j^2$ or $x_ix_j$. We distinguish two subcases.

\textbf{Subcase 2.1.} For any $1\leq i<j\leq n+1$,
there is $\bw_{ij}\in S_2(D_\bq(\bu))$ such that $\bw_{ij}=x_i^2$ or $\bw_{ij}=x_j^2$.
One can deduce that $\{i\mid 1\leq i\leq n+1, x_i^2\in S_2(D_\bq(\bu))\}$
contains at least $n$ elements.
We may assume that it contains $1, 2, \ldots, n$.
Then $x_1^2\bp_1, x_2^2\bp_2, \dots, x_{n}^2\bp_{n}\in D_\bq(\bu)$ for some $\bp_1, \bp_2, \ldots, \bp_n \in X^*$.
Since $c(\bq)=c(D_\bq(\bu))$,
it follows that there exists $\bu_j\in D_\bq(\bu)$ such that $\bu_j=x_{n+1}\bp_{n+1}$ for some $\bp_{n+1}\in X^*$.
Now we have
\begin{align*}
\bu
&\succeq x_1^2\bp_1+x_2^2\bp_2+\dots+x_{n}^2\bp_{n}+x_{n+1}\bp_{n+1}\\
&\succeq x_1^2\bp'_1+x_2^2\bp'_2+\dots+x_{n}^2\bp'_{n}+x_{n+1}\bp'_{n+1}&&(\text{by}~\eqref{54503})\\
&\succeq x_1x_2\cdots x_{n+1}=\bq, &&(\text{by}~\eqref{63404},\eqref{63406})
\end{align*}
where $\bp'_i$ is either empty or linear, $1\leq i \leq n+1$.
This implies the inequality $\bu \succeq \bq$.

\textbf{Subcase 2.2.} There exist $1\leq i<j\leq n+1$ such that
for every $\bw\in S_2(D_\bq(\bu))$, $\bw\neq x_{i}^2$ and $\bw\neq x_{j}^2$.
Consequently, $x_{i}x_{j} \in S_2(D_\bq(\bu))$.
Hence $x_{i}x_{j}\bp_{ij}\in D_\bq(\bu)$ for some $\bp_{ij}\in X^*$.

For convenience, we assume that for some integer $k$ with $0\leq k<n$,
$x_i^2\in S_2(\bu)$ for every $0\leq i\leq k$, and $x_ix_j\in S_2(\bu)$ for every $k+1\leq i<j\leq n+1$.
Then $x_i^2\bp_i\in D_\bq(\bu)$ for some $\bp_i\in X^*$, $0\leq k<n$;
$x_ix_j\bp_{ij}\in D_\bq(\bu)$ for some $\bp_{ij}\in X^*$, $k+1\leq i<j\leq n+1$.
Now
\begin{align*}
\bu
&\succeq x_1^2\bp_1+x_2^2\bp_2+\dots+x_k^2\bp_k+\sum_{k+1\leq i<j\leq n+1}x_ix_j\bp_{ij}\\
&\succeq x_1^2\bp'_1+x_2^2\bp'_2+\dots+x_k^2\bp'_k+\sum_{k+1\leq i<j\leq n+1}x_ix_j\bp'_{ij} &&(\text{by}~\eqref{54503})\\
&\succeq x_1x_2\cdots x_{n+1}=\bq, &&(\text{by}~\eqref{63405}, \eqref{63407})
\end{align*}
where $\bp'_i$ is either empty or linear for every $1\leq i \leq k$,
and $\bp'_{ij}$ is either empty or linear for every $k+1\leq i<j\leq n+1$.
This derives the inequality $\bu \succeq \bq$.
\end{proof}

\begin{thm}
The ai-semiring $S_{(4, 634)}$ is nonfinitely based.
\end{thm}
\begin{proof}
By Proposition~\ref{pro63401},
\[
\Xi=\{\eqref{54501}, \eqref{54503}, \eqref{63404}, \eqref{63406}, \eqref{63407}\}
\cup\{\delta_{n, \bv} \mid n\geq 2, \bv\in \Theta_n\}
\]
is an equational basis for the commutative ai-semiring $S_{(4, 634)}$, which contains $\Omega$.
By Theorem~\ref{thm26012501}, to show that $S_{(4, 634)}$ is nonfinitely based,
it suffices to prove that for any $m\geq 3$, $\bu_{m, m}$ is $\bw$-free for every term $\bw$ in the sets
\[
\Gamma=\{\bt \mid \bt ~\text{is an upper side of an inequality in}~\{\eqref{54501}, \eqref{54503},\eqref{63404},\eqref{63406},\eqref{63407}\}\}
\]
and
\[
\{\bv \mid \bv\in \Theta_n\setminus \{\bu_{n,n}\}, n\geq 2\}.
\]
Indeed, the required freeness follows from Lemma~\ref{lem26011801}(d) for $\mathbf{w} \in\Gamma$
and from Proposition~\ref{pro26012502} for $\mathbf{w}\in \{\bv \mid \bv\in \Theta_n\setminus  \{\bu_{n,n}\}, n\geq 2\}$.
\end{proof}

\section{The interval $[\mathsf{V}(S_{(4, 545)}), \mathsf{V}(S_{(4, 634)})]$}
In this section, we prove that the interval
$[\mathsf{V}(S_{(4, 545)}), \mathsf{V}(S_{(4, 634)})]$,
which consists of all subvarieties of $\mathsf{V}(S_{(4, 634)})$ that include $\mathsf{V}(S_{(4, 545)})$,
contains $2^{\aleph_0}$ distinct varieties.
Consequently, the interval $[\mathsf{V}(S_{53}), \mathsf{V}(S_{53}^0)]$ also comprises $2^{\aleph_0}$ distinct varieties.

To better understand the interval,
we first provide an equational characterization of its lower end $\mathsf{V}(S_{(4, 545)})$ in the upper end $\mathsf{V}(S_{(4, 634)})$.
Although the explicit description itself is not directly used in the later cardinality argument,
it gives a first indication of how much $\mathsf{V}(S_{(4, 634)})$ exceeds $\mathsf{V}(S_{(4, 545)})$.

\begin{pro}\label{lem05}
$\mathsf{V}(S_{(4, 545)})$ is the subvariety of $\mathsf{V}(S_{(4, 634)})$ defined by the inequalities \eqref{54504} and
$\sigma_{n,k}$ for all $n \geq 1$, $0\leq k<n$.
\end{pro}
\begin{proof}
By Corollary~\ref{coro26012301}, $\mathsf{V}(S_{(4, 545)})$ is the subvariety for $\mathsf{V}(S_{(4, 634)})$.
Proposition~\ref{pro54501} shows that the identities \eqref{54501}, \eqref{54503}, \eqref{54504}
together with all $\sigma_{n,k}; (n \geq 1, 0\leq k\leq n)$
form an equational basis for $\mathsf{V}(S_{(4, 545)})$.
By Proposition~\ref{pro63401}, the identities \eqref{54501}, \eqref{54503}, \eqref{63404}--\eqref{63407}
form an equational basis for $\mathsf{V}(S_{(4, 634)})$.
So it suffices to prove that $\sigma_{n,n}$ is derivable
from the above equational basis for $\mathsf{V}(S_{(4, 634)})$.
Indeed, choose an inequality in \eqref{63405}, where
\[
\bw_{ij}=x_1\cdots x_{i-1}x_{i+1}\cdots x_{j-1}x_{j+1}\cdots x_n
\]
if $1\leq i<j\leq n$,
and $\bw_{ij}=\varepsilon$ otherwise. This derives the identity
\[
\prod_{i=1}^{n+1} x_i \preceq \prod_{i=1}^n x_i+\left(\sum_{i=1}^nx_i\right)x_{n+1},
\]
which is precisely $\sigma_{n,n}$.
\end{proof}

The proof of the following result follows the same pattern as that of Proposition~\ref{pro26012502}.
\begin{pro}\label{lem06}
Let $n\geq2$ be an integer and let $\bv$ be a term in $\Theta_n$. Then $\bu_{m,0}$ is $\bv$-free for every $m\geq 1$.
\end{pro}
\begin{proof}
Now suppose for contradiction that $\bu_{m, 0}$ is not $\bv$-free for some $\bv\in \Theta_n$ and $n\geq2$.
Then there exist terms $\bp, \br\in P_f(X_c^+)$ and a substitution $\varphi$ such that
\begin{equation}\label{26012101}
\bp\varphi(\bv)+\br=\bu_{m,0},
\end{equation}
where $\bp$ may be the empty word, and $\br$ may be the empty set.

\begin{claim}\label{claimqj01}
For distinct words $x_{i_1}x_{j_1}\bw_{i_1j_1}$ and $x_{i_2}x_{j_2}\bw_{i_2j_2}$ in $\bv$,
we have
\[
\bp\varphi(x_{i_1}x_{j_1}\bw_{i_1j_1})\neq\bp\varphi(x_{i_2}x_{j_2}\bw_{i_2j_2}).
\]
\end{claim}
\begin{proof}[Proof of Claim $\ref{claimqj01}$.]
The argument is identical to the proof of Claim~\ref{claim63401}.
\end{proof}

\begin{claim}\label{claimqj03}
The term $\bp$ is a word.
\end{claim}
\begin{proof}[Proof of Claim $\ref{claimqj03}$.]
Suppose that $\bp_1$ and $\bp_2$ are distinct words in $\bp$.
By \eqref{26012101}, both $\bp_1\varphi(x_1x_2\bw_{12})$ and $\bp_2\varphi(x_1x_2\bw_{12})$ are contained in $\bu_{m,0}$.
Consequently, $\bu_{m,0}$ would contain two distinct words sharing a common subword of length $2$,
which contradicts Lemma~\ref{lem26011801}(c).
\end{proof}

\begin{claim}\label{claimqj04}
$\bp$ is either empty or a single variable.
\end{claim}
\begin{proof}[Proof of Claim $\ref{claimqj04}$.]
By Claim~\ref{claimqj03}, we know that $\bp$ is a word.
If its length exceeds $1$, then by \eqref{26012101} and Lemma~\ref{lem26011801}(g),
for distinct words $x_{i_1}x_{j_1}\bw_{i_1j_1}$ and $x_{i_2}x_{j_2}\bw_{i_2j_2}$ in $\bv$,
we have
\[
\bp\varphi(x_{i_1}x_{j_1}\bw_{i_1j_1})=\bp\varphi(x_{i_2}x_{j_2}\bw_{i_2j_2})=x_1x_2\cdots x_m.
\]
This contradicts Claim~\ref{claimqj01}.
Thus the length of $\bp$ is at most $1$;
therefore $\bp$ is either empty or a single variable.
\end{proof}

\begin{claim}\label{claimqj05}
The word $\bp$ is empty.
\end{claim}
\begin{proof}[Proof of Claim $\ref{claimqj05}$.]
Suppose for contradiction that $\bp$ is nonempty.
By Claim~\ref{claimqj04}, $\bp$ is then a single variable.
For distinct words $x_{i_1}x_{j_1}\bw_{i_1j_1}$ and $x_{i_1}x_{j_2}\bw_{i_1j_2}$ in $\bv$,
Claim~\ref{claimqj01} tells us that
\[
\bp\varphi(x_{i_1}x_{j_1}\bw_{i_1j_1})\neq\bp\varphi(x_{i_1}x_{j_2}\bw_{i_1j_2}).
\]

Observe that $\bp\varphi(x_{i_1}x_{j_1}\bw_{i_1j_1})$ and $\bp\varphi(x_{i_1}x_{j_2}\bw_{i_1j_2})$ share a common subword of length $2$,
which contradicts Lemma~\ref{lem26011801}(c).
Thus $\bp$ is empty.
\end{proof}

By Claim~\ref{claimqj05}, the equality~\eqref{26012101} reduces to the form
\begin{equation}\label{26012102}
\varphi(\bv)+\br=\bu_{m, 0}.
\end{equation}

\begin{claim}\label{claimqj02}
The term $\bv$ contains exactly one word of length at least $3$.
\end{claim}
\begin{proof}[Proof of Claim $\ref{claimqj02}$.]
By definition, every word in $\bv$ has length at least $2$.
We argue by contradiction.

\textbf{Case 1.}
All words in $\bv$ have length $2$. Then
$\bw_{ij}=\varepsilon$ for all $1\leq i<j\leq n+1$, and consequently
\[
\varphi(x_1)\varphi(x_2)+\varphi(x_1)\varphi(x_3)+\varphi(x_2)\varphi(x_3)\subseteq \bu_{m,0},
\]
which contradicts Lemma~\ref{lem26011801}(j).

\textbf{Case 2.}
$x_{i_1}x_{j_1}\bw_{i_1j_1}$ and $x_{i_2}x_{j_2}\bw_{i_2j_2}$
are distinct words in $\bv$ with length at least $3$.
Then both $\bw_{i_1j_1}$ and $\bw_{i_2j_2}$ are nonempty.
By Claim~\ref{claimqj01},
$\varphi(x_{i_1}x_{j_1}\bw_{i_1j_1})$ and $\varphi(x_{i_2}x_{j_2}\bw_{i_2j_2})$
cannot be simultaneously equal to $x_1x_2\cdots x_m$, and so
$x_iy_ix_{m+1}\in \varphi(x_{i_1}x_{j_1}\bw_{i_1j_1})$ or $x_iy_ix_{m+1}\in \varphi(x_{i_2}x_{j_2}\bw_{i_2j_2})$ for some $1\leq i\leq m$.
We may assume that $x_iy_ix_{m+1}\in \varphi(x_{i_1}x_{j_1}\bw_{i_1j_1})$.
If $\varphi(x_{i_1}x_{j_1}\bw_{i_1j_1})$ contained another word,
it could not be expressed as a product of three terms; hence it contains no other word.
Consequently,
$\varphi(x_{i_1}x_{j_1}\bw_{i_1j_1})=x_iy_ix_{m+1}$, and so $\bw_{i_1j_1}$ is a variable.
Suppose without loss of generality that
$\varphi(x_{i_1})=x_i$, $\varphi(x_{j_1})=y_i$, and $\varphi(\bw_{i_1j_1})=x_{m+1}$.

If $\varphi(x_{i_2}x_{j_2}\bw_{i_2j_2})=x_1x_2\cdots x_m$,
then there exists $x_s\in c(x_{i_2}x_{j_2}\bw_{i_2j_2})$ such that $\varphi(x_s)\neq x_i$.
Since $x_s\neq x_{j_1}$, it follows that $x_sx_{j_1}$ is a subterm of $\bv$.
This implies that $\varphi(x_sx_{j_1})$ is a subterm of $\bu_{m,0}$,
and $x_ry_i$ is a subterm of $\varphi(x_sx_{j_1})$ for some $1\leq r\leq m$, $r\neq i$.
Thus $x_ry_i$ is a subterm of $\bu_{m,0}$, a contradiction.

If $\varphi(x_{i_2}x_{j_2}\bw_{i_2j_2})\neq x_1x_2\cdots x_m$, then $\varphi(x_{i_2}x_{j_2}\bw_{i_2j_2})=x_jy_jx_{m+1}$
for some $1\leq j\leq m$ and $i\neq j$.
Consequently, $y_iy_j$ would be a subword of $\varphi(x_{j_1}x_{i_2})$, $\varphi(x_{j_1}x_{j_2})$, or $\varphi(x_{j_1}\bw_{i_2j_2})$.
Therefore, $y_iy_j\leq \bu_{m, 0}$, which is impossible.
\end{proof}

Now we complete the main proof of the proposition.
By  Lemma~\ref{lem26011801}(e), \eqref{26012102} and Claim~\ref{claimqj02}, we may write
\[
\bv=\prod_{i=1}^k x_i+\sum_{1\leq i\leq k, k+1\leq j\leq n+1}x_ix_j+\sum_{k+1\leq i<j\leq n+1}x_ix_j,
\]
where $3\leq k\leq n$.

If $\varphi(x_1x_2\cdots x_k)=x_1x_2\cdots x_m$, then Claim~\ref{claimqj01} tells us that
$x_iy_ix_{m+1}$ is a word in $\varphi(x_1x_{k+1})$ for some $1\leq i\leq m$.
This implies that $\varphi(x_1)=x_i$, and so
\[
c(\varphi(x_2\cdots x_k))\subseteq\{x_1,\cdots, x_{i-1}, x_{i+1},\cdots, x_m\}.
\]
Moreover, $y_ix_{m+1}$ is a word in $\varphi(x_{k+1})$.
It follows that there exists $j$ $(1\leq j\leq m, j\neq i)$ such that
$x_jy_ix_{m+1}\leq\varphi(x_2x_{k+1})$, and so
$x_jy_ix_{m+1} \leq \bu_{m, 0}$, a contradiction.

If $\varphi(x_1x_2\cdots x_k)\neq x_1x_2\cdots x_m$,
then $\varphi(x_1x_2\cdots x_k)=x_iy_ix_{m+1}$ for some $1\leq i\leq m$, which forces $k=3$.
By Claim~\ref{claimqj01}, we have $\varphi(x_1x_4)\neq\varphi(x_2x_4)$.
Hence, for some $j$ $(1\leq j\leq m, j\neq i)$, $x_jy_jx_{m+1}$ lies in either $\varphi(x_1x_4)$ or $\varphi(x_2x_4)$.
Assume without loss of generality that $x_jy_jx_{m+1}$ is a word in $\varphi(x_1x_4)$.
This yields $\varphi(x_1)=x_{m+1}$ and $x_jy_j$ is a word in $\varphi(x_4)$.
Therefore, either $x_ix_jy_j$ or $y_ix_jy_j$ must be a word in $\varphi(x_2x_4)$.
Thus either $x_ix_jy_j$ or $y_ix_jy_j$ would be a word in $\bu_{m, 0}$, which is impossible.

This completes the proof.
\end{proof}

\begin{thm}\label{thm01}
The interval $[\mathsf{V}(S_{(4, 545)}), \mathsf{V}(S_{(4, 634)})]$ contains $2^{\aleph_0}$ distinct varieties.
\end{thm}
\begin{proof}
Let $\mathbb{N}$ denote the set of all positive integers.
For each subset $M$ of $\mathbb{N}$,
let $\mathcal{V}_M$ denote the subvariety of $\mathsf{V}(S_{(4, 634)})$
defined by the identities $\sigma_{m, 0}$ with $m\in M$.
Since every inequality of the form $\sigma_{m, 0}$ holds in $S_{(4, 545)}$,
each $\mathcal{V}_M$ lies in the interval $[\mathsf{V}(S_{(4, 545)}), \mathsf{V}(S_{(4, 634)})]$.

Now take two subsets $P$ and $Q$ of $\mathbb{N}$.
We claim that $\mathcal{V}_P=\mathcal{V}_Q$ if and only if $P=Q$.
Suppose that $P\neq Q$; without loss of generality let $p \in P \setminus Q$.
It is evident that $\mathcal{V}_P$ satisfies $\sigma_{p, 0}$.
To see that $\mathcal{V}_Q$ does not satisfy $\sigma_{p,0}$, recall that $\mathcal{V}_Q$ is defined by
the identities \eqref{54501}, \eqref{54503}, \eqref{63404}--\eqref{63407} together with
all $\sigma_{q,0}\;(q \in Q)$.
By Lemma~\ref{lem26012701} together with Lemma~\ref{lem02} and
Propositions~\ref{pro26012501}, \ref{pro63401} and \ref{lem06},
the upper side of $\sigma_{p,0}$
is free for the upper side of each of the defining inequalities of $\mathcal{V}_Q$.
Consequently, $\sigma_{p,0}$ cannot hold in $\mathcal{V}_Q$.
Therefore, $\mathcal{V}_P \neq \mathcal{V}_Q$.

Thus $\{\mathcal{V}_N \mid N\subseteq \mathbb{N}\}$ has the same cardinality as the power set of $\mathbb{N}$, and so
the interval $[\mathsf{V}(S_{(4, 545)}), \mathsf{V}(S_{(4, 634)})]$ contains $2^{\aleph_0}$ distinct varieties.
\end{proof}

\begin{cor}
There are $2^{\aleph_0}$ distinct nonfinitely based varieties in the interval $[\mathsf{V}(S_{(4, 545)}), \mathsf{V}(S_{(4, 634)})]$.
\end{cor}
\begin{proof}
From the proof of Theorem~\ref{thm01},
it is enough to show that for any infinite subset $M$ of $\mathbb{N}$,
the subvariety $\mathcal{V}_M$ of $\mathsf{V}(S_{(4,634)})$ defined by the identities $\sigma_{m,0}$ with $m \in M$
is nonfinitely based.
Indeed, by Proposition~\ref{pro63401},
\[
\Sigma=\{\eqref{54501}, \eqref{54503}, \eqref{63404}, \eqref{63405}, \eqref{63406}, \eqref{63407}\}
\cup\{\sigma_{m,0} \mid m \in M\}
\]
is an equational basis for the commutative ai-semiring variety $\mathcal{V}_M$, 
which contains an infinite subset of $\Omega_0$.
According to Theorem~\ref{thm26012401}, it suffices to prove that 
$\bu_{m, 0}$ is $\bt$-free for every inequality $\bs \preceq \bt$
in $\Sigma \setminus \Omega_0$ and for every $m\geq 4$. 
Observe that 
\[
\Sigma \setminus \Omega_0=\{\eqref{54501}, \eqref{54503}, \eqref{63404}, \eqref{63405}, \eqref{63406}, \eqref{63407}\}.
\]
The required freeness follows from Lemma~\ref{lem26011801}(d) for $\mathbf{t} \in\Gamma$, where
\[
\Gamma=\{\bw \mid \bw ~\text{is an upper side of an inequality in}~\{\eqref{54501}, \eqref{54503},\eqref{63404},\eqref{63406},\eqref{63407}\}\}
\]
and from Proposition~\ref{lem06} for $\mathbf{t}\in \{\bv \mid \bv\in \Theta_n, n\geq 2\}$.
\end{proof}

\begin{cor}\label{cor02}
The interval $[\mathsf{V}(S_{53}), \mathsf{V}(S_{53}^0)]$ contains $2^{\aleph_0}$ distinct varieties.
\end{cor}
\begin{proof}
Since $S_{53}^0$ is isomorphic to $S_{(4, 634)}$, and $S_{53}$ belongs to $\mathsf{V}(S_{(4, 545)})$,
it follows that
the interval $[\mathsf{V}(S_{(4, 545)}), \mathsf{V}(S_{(4, 634)})]$
is contained in $[\mathsf{V}(S_{53}), \mathsf{V}(S_{53}^0)]$.
Theorem~\ref{thm01} shows that the former already contains \(2^{\aleph_0}\) varieties,
hence the latter must also contain \(2^{\aleph_0}\) distinct varieties.
\end{proof}

\section{Conclusion}
In this paper, we have provided two new sufficient conditions for an ai-semiring to be nonfinitely based,
and applied them to show that the $4$-element ai-semirings $S_{(4, 545)}$ and $S_{(4, 634)}$ have no finite equational basis.
As a consequence, the variety $\mathsf{V}(S_{(4, 545)})$ is the first known example of a nonfinitely based ai-semiring variety
that can be written as the join of two finitely based varieties.

Both $\mathsf{V}(S_{(4, 545)})$ and $\mathsf{V}(S_{(4, 634)})$ contain the $3$-element ai-semiring $S_7$;
hence the solution of their finite basis problem provides direct support
for the affirmative side of the question raised by Gao et al.~\cite{gjrz2},
which asked whether every finite ai-semiring whose variety contains $S_7$ is nonfinitely based.

Given an ai-semiring $S$ whose additive top element $\infty$ is also a multiplicative zero,
one can adjoin a new element $e$ and define
\[
e+a = a+e = e+e = ee = e, \quad ea = ae = \infty \quad (a \in S \setminus\{e\}).
\]
The resulting algebra $(S \cup\{e\},+,\cdot)$, called the \emph{idempotent extension} of $S$ and denoted by $S^e$,
is isomorphic to a subdirect product of $S$ and $D_2$;
hence $\mathsf{V}(S^e)=\mathsf{V}(S,D_2)$.
In particular, $\mathsf{V}(S_{53}^e)=\mathsf{V}(S_{53},D_2)$.
By Proposition~\ref{pro26011401}, this variety coincides with
$\mathsf{V}(S_{(4,545)})$. Therefore, $S_{53}^e$ is also nonfinitely based,
showing that the idempotent extension construction does not preserve the finite basis property.

On the other hand, $S_{(4, 634)}$ is isomorphic to $S_{53}^0$; thus $S_{53}^0$ is nonfinitely based as well.
This yields the smallest example of a finitely based ai-semiring $S$
whose extension $S^0$ is nonfinitely based.

Combined with earlier results obtained by members of our group,
the present paper advances the finite basis problem for all $4$-element ai-semirings
whose additive reducts are chains within reach.
A complete solution of the finite basis problem for this subclass will be presented in a forthcoming work \cite{ryy}.

Finally, we expect that the syntactic methods developed here, as well as the explicit sufficient conditions obtained,
will find further applications in the study of finite basis questions for other classes of ai-semirings.

\end{document}